\documentclass[11pt,a4paper]{amsart}

\usepackage[english]{babel}
\usepackage{amsmath,amsthm,amssymb,amsfonts}
\usepackage[pdftex]{color}
\usepackage[dvipsnames]{xcolor}
\usepackage[bookmarks=true,hyperindex,pdftex,colorlinks,citecolor=blue,linkcolor=blue,urlcolor=blue]{hyperref}
\usepackage{graphicx}
\usepackage{mathtools}
\usepackage{mathabx}
\usepackage{cancel}

\makeatletter
\@namedef{subjclassname@2020}{%
  \textup{2020} Mathematics Subject Classification}
\makeatother

\newtheorem{theorem}{Theorem}[section]
\newtheorem{lemma}[theorem]{Lemma}

\newtheorem{proposition}[theorem]{Proposition}
\newtheorem{corollary}[theorem]{Corollary}
\newtheorem{question}[theorem]{Question}

\theoremstyle{definition}

\newtheorem{example}[theorem]{Example}
\newtheorem{remark}[theorem]{Remark}
\theoremstyle{remark}

\numberwithin{equation}{section}

\newcommand{\sna}{\operatorname{SNA}}
\newcommand{\na}{\operatorname{NA}}
\newcommand{\lip}{\mathrm{Lip}_0}
\newcommand{\Mol}{\mathrm{Mol}}
\newcommand{\cconv}{\overline{\mathrm{conv}}}
\newcommand{\N}{{\mathbb{N}}}

\title[Linear spaces of SNA Lipschitz functionals]{Closed linear spaces consisting of strongly norm attaining Lipschitz functionals}

\author[V.\ Kadets]{Vladimir Kadets}
\address[Vladimir Kadets]{School of Mathematics and Computer Sciences, V.N. Karazin Kharkiv National University, 4 Svobody Sq, 61022 Kharkiv, Ukraine.
\href{https://orcid.org/0000-0002-5606-2679}{ORCID: \texttt{0000-0002-5606-2679} }}
\email{v.kateds@karazin.ua}

\author[\'O.\ Rold\'an]{\'Oscar Rold\'an}
\address[\'Oscar Rold\'an]{Departamento de An\'{a}lisis Matem\'{a}tico,
Universidad de Valencia, Doctor Moliner 50, 46100 Burjasot (Valencia), Spain.
\href{https://orcid.org/0000-0002-1966-1330}{ORCID: \texttt{0000-0002-1966-1330} }}
\email{oscar.roldan@uv.es}

\thanks{The first author was supported by the National Research Foundation of Ukraine funded by Ukrainian State budget as part of the project 2020.02/0096 ``Operators in infinite-dimensional spaces: the interplay between geometry, algebra and topology''. The second author was supported by the Spanish Ministerio de Universidades, grant FPU17/02023, and by the Ministerio de Ciencia e Innovación and Fondo Europeo de Desarrollo Regional project MTM2017-83262-C2-1-P / MCIN / AEI / 10.13039 / 501100011033 (FEDER)}


\date{\today}


\begin{document}
	
\begin{abstract}
Given a pointed metric space $M$, we study when there exist $n$-dimensional linear subspaces of $\lip(M)$ consisting of strongly norm-attaining Lipschitz functionals, for $n\in\mathbb{N}$. We show that this is always the case for infinite metric spaces, providing a definitive answer to the question. We also study the possible sizes of such infinite-dimensional closed linear subspaces $Y$, as well as the inverse question, that is, the possible sizes of the metric space $M$ given that such a subspace $Y$ exists. We also show that if the metric space $M$ is $\sigma$-precompact, then the aforementioned subspaces $Y$ need to be always separable and isomorphically polyhedral, and we show that for spaces containing $[0,1]$ isometrically, they can be infinite-dimensional.
\end{abstract}

\maketitle
\section{Introduction}\label{section:introduction}

\subsection{Motivation}\label{subsection:motivation}
According to Rmoutil's result \cite{Rmoutil17}, there exists an infinite-dimensional Banach space $X$ (namely $c_0$ in the equivalent norm constructed by Read \cite{Read}) such that the set  NA$(X) \subset X^*$ of norm attaining linear functionals does not contain two-dimensional linear subspaces. That was a negative answer to \cite[Problem III]{Godefroy01} by Godefroy. Read's construction was generalized in \cite{KLMD}, where such equivalent norms with ``extremely nonlineable set of norm attaining functionals'' were constructed in a number of other Banach spaces, in particular in all separable and some non-separable Banach spaces containing $c_0$.

We address an analogous question for metric spaces $M$ and the set $\sna(M)$ of strongly norm attaining Lipschitz functionals. Surprisingly for the authors, for Lipschitz functionals the answer happens to be just the opposite: for every infinite $M$ the corresponding set $\lip(M)$ always has linear subspaces of dimension at least 2 consisting of strongly norm-attaining functionals, and in fact, it contains such subspaces of arbitrarily big finite dimension. After figuring out this new fact, we study some other natural questions about possible sizes of closed linear subspaces in $\sna(M)$.

It was shown in \cite[Theorem 3.2]{CDW16} that if $M$ is an infinite metric space, then $\lip(M)$ contains linear subspaces isomorphic to $\ell_\infty$, and later, an isometric version of this result was given in \cite[Theorem 5]{CJ17}. However, the proofs cannot be adapted to the setting of strongly norm-attaining Lipschitz mappings in general and, as we show in Theorem \ref{theorem:inverse-problem}, for separable $M$, such a big subspace as  $\ell_\infty$ cannot be squeezed in $\sna(M)$.

\subsection{Notation and preliminaries}

Let us quickly recall some standard notation from Banach space theory (see \cite{FHHMZ11} for instance). All vector spaces in this document are assumed to be real. Let $X$ be a Banach space (over the field $\mathbb{R}$ of real numbers). We denote by $X^*$, $B_X$ and $S_X$ its topological dual, its closed unit ball, and its unit sphere, respectively. The notations $\mathrm{conv}(A)$ and $\mathrm{aconv}(A)$ respectively mean the convex hull and the absolute convex hull of $A$. Their closures with respect to a topology $\tau$ are denoted by $\overline{\mathrm{conv}}^\tau(A)$ and $\overline{\mathrm{aconv}}^\tau(A)$, respectively, and when $\tau$ is the usual topology, it will be omitted. If $X$ and $Y$ are two Banach spaces, $\mathcal{L}(X, Y)$ is the space of bounded linear operators from $X$ to $Y$, and $\operatorname{NA}(X, Y)$ is the subset of $\mathcal{L}(X, Y)$ consisting of all those operators that attain their norms. We will use  without additional explanation the standard notation like $c_0$, $\ell_1$, $\ell_p$, $\ell_\infty$, $L_\infty[0,1]$ for corresponding classical Banach spaces. For $n\in\mathbb{N}$, $\ell_1^n$ denotes the $n$-dimensional version of $\ell_1$, that is the space $\mathbb{R}^n$ endowed with the norm 
$$
\|(x_1, \ldots, x_n)\|_1 = \sum_{k=1}^n |x_k|,
$$
and $\ell_\infty^n$ is defined analogously with the norm $\|\cdot\|_\infty$,
 $$
\|(x_1, \ldots, x_n)\|_\infty = \max_{1\le k \le n} |x_k|.
$$
Recall also that $\ell_1^2$ is isometrically isomorphic to $\ell_\infty^2$.

We will also use the following concepts. Given a Banach space $X$, a set $B\subset B_{X^*}$ is a \textit{James boundary of $X$} if for every $x\in X$, there is $g\in B$ such that $g(x)=\|x\|$ (see \cite[Definition 3.118]{FHHMZ11}). A metric space is said to be \textit{$\sigma$-precompact} if it is a countable union of precompact sets, and a metric space $M$ is said to have the \textit{small ball property} if for every $\varepsilon_0>0$, it is possible to write $M$ as a union of a sequence $(B(x_n, r_n))_n$ of closed balls such that $r_n<\varepsilon_0$ for all $n$ and $r_n\xrightarrow[n\to\infty]{} 0$. It is known that $\sigma$-precompact spaces have the small ball property but the converse is not true in general (see \cite[Theorem 5.2]{BeKa01}). Finally, recall that a Banach space $X$ is \textit{polyhedral} if the unit ball of every finite-dimensional subspace of $X$ is a polytope. A space that is isomorphic to a polyhedral space is said to be \emph{isomorphically polyhedral}.

Let $(M,\rho)$ be a pointed metric space (that is, a metric space consisting of at least 2 points and containing a distinguished point $0$). We will usually consider only one metric on $M$ which permits to write just $M$ for the metric space instead of $(M,\rho)$. We use the standard notation $\lip(M)$ for the space of all Lipschitz mappings $f:M\rightarrow \mathbb{R}$ such that $f(0)=0$ endowed with the Lipschitz constant as the norm, that is
$$\|f\| = \sup \left\{  \frac{\left|f(x)-f(y)\right|}{\rho(x,y)}  :\, x,y\in M,\, x\neq y\right\}.$$
The interested reader can find a detailed study of Lipschitz spaces in the book \cite{Weaver99}.

There is a natural way to define norm-attainment in this context. According to \cite{KMS16}, a Lipschitz functional $f\in\lip(M)$ is said to \textit{attain its norm strongly} if there is a pair of points $x,y\in M$ with $x\neq y$ such that $$\|f\| = \frac{\left| f(x)-f(y)\right|}{\rho(x,y)} .$$

We will denote the set of Lipschitz functionals from $\lip(M)$ that attain their norm strongly by $\sna(M)$ (the notations $\operatorname{SA}(M)$ and $\operatorname{LipSNA}(M)$ have also been used in the literature). 

The reason behind calling this natural norm-attainment \textit{strong} is that this is a restrictive notion, and other weaker notions of norm-attainment, which are also natural and give interesting results, have also been introduced and studied since the initial works on the topic \cite{Godefroy16,KMS16} (see also \cite[Section 1]{CCM20} for a very clean exposition of various kinds of norm-attainment for Lipschitz mappings and the relations between them).

The systematic study of norm-attaining Lipschitz mappings was started in \cite{Godefroy16} and \cite{KMS16}. Since then, a fruitful line of research arose and continues to be very active nowadays. As we just mentioned, the notion of strong norm-attainment is a bit restrictive. This can be justified by the following facts:
\begin{itemize}
\item If a Lipschitz functional $f$ attains its norm at some pair of points $x\neq y$, then it also has to attain its norm at any pair of different points in between them (see \cite[Lemma 2.2]{KMS16} for the details).
\item If $M$ is a metric length space (that is, if for every $x\neq y\in M$, the distance $\rho(x,y)$ is equal to the infimum of the length of rectifiable curves joining them; note that every normed space is a metric length space), then $\sna(M)$ is never dense in $\lip(M)$ (for the details, see \cite[Theorem 2.2]{CCGMR19}, which improves \cite[Theorem 2.3]{CCGMR19}, and check also \cite{AMC19, GLPRZ18, IKW07} for more background and characterizations on metric length spaces).
\end{itemize}
However, despite that, positive results have also been achieved in this direction in the recent years for some metric spaces (see for instance \cite[Section 3]{CCGMR19} and \cite{CM19}). Evidently, if $M$ is finite, then $\sna(M) = \lip(M)$ is a linear space, so we are mainly interested in infinite metric spaces.

An important tool in the study of Lipschitz mappings is the concept of \textit{Lipschitz-free spaces} (also referred to as \textit{Arens-Eells spaces} and \textit{Transportation cost spaces} in the literature). Given a metric space $M$, denote $\delta:M\rightarrow (\lip(M))^*$ the canonical embedding given by $\delta(x) = \delta_x$, $x\in M$, where $\delta_x$ is the evaluation functional $f \mapsto f(x)$. Then the norm-closed linear span  $\mathcal{F}(M)$ of $\delta(M)$ in $\lip(M)$  is called the Lipschitz-free space over $M$. The space $\mathcal{F}(M)$ can be seen isometrically as a predual of $\lip(M)$ (see \cite[Section 1]{CCGMR19} and the survey \cite{Godefroy15} for a solid background on Lipschitz-free spaces). The identification $(\mathcal{F}(M))^* = \lip(M)$ can be explained as follows: every Lipschitz mapping $f:M\rightarrow \mathbb{R}$ can be identified with the continuous linear mapping $\hat{f}:\mathcal{F}(M)\rightarrow \mathbb{R}$ given by $\hat{f}(\delta_p)\mapsto f(p)$ for $p\in M$ and extended to the whole $\mathcal{F}(M)$ by linearity and continuity. This identifies isometrically the spaces $\lip(M)$ and $\mathcal{L}(\mathcal{F}(M), \mathbb{R}) = (\mathcal{F}(M))^*$. It is easy to check that $\sna(M)$ can be  identified with the set of those elements of $\mathcal{L}(\mathcal{F}(M), \mathbb{R})$ that attain their norm at a point of the form $\frac{\delta_x-\delta_y}{\rho(x, y)}$, for $x\neq y\in M$. This identification has been used to get many results about strongly norm-attaining Lipschitz mappings (see for instance \cite[Section 3]{CCGMR19}, \cite[Section 2]{CGMR21}, \cite{CM19}, \cite[Section 7]{GPPR18} and \cite[Section 4]{GPR17}). We refer to \cite{CJ17,KMO20,OO20} for works where the possibility to embed $\ell_1$ into Lipschitz-free spaces was studied in depth.

Remark also, that the structure of $\lip(M)$ and $\sna(M)$ does not depend on the selection of the distinguished point $0$: if $M'$ is the same metric space but with another distinguished point $0'$ then the mapping $f \mapsto f - f(0')$ is a bijective linear isometry between $\lip(M)$ and $\lip(M')$, which maps $\sna(M)$ to $\sna(M')$.

Another important tool is the well known \textit{McShane's extension theorem} \cite[Theorem 1.33]{Weaver99} that allows us to extend $f\in \lip(M_1)$ to $\tilde{f}\in\lip(M_2)$, with $M_1\subset M_2$, in such a way that $\|f\|=\|\tilde{f}\|$.

Given a metric space $M$, in this text, the expression \textit{linear subspaces of $\sna(M)$} should be understood as linear subspaces of $\lip(M)$ consisting of strongly norm-attaining Lipschitz functionals. Also, we use below the following slang. Let $Y$ be a Banach space and $M$ be a pointed metric space. We say that $Y$ embeds in  $\sna(M)$ (or equivalently $\sna(M)$ contains a copy of $Y$), if there is a linear isometric embedding  $U: Y \to \lip(M)$ such that $U(Y) \subset \sna(M)$
 
\subsection{The structure of the article}
The rest of the paper is structured in 3 sections as follows. 

In Section \ref{section:finite-dim-subs}, we ask for what metric spaces $M$ there exist $n$-dimensional linear subspaces inside of $\sna(M)$, for $n\in\mathbb{N}$. The main result of the section, Theorem \ref{theorem:finite-dimensional-subspace}, asserts that if $M$ contains at least $2^n$ points, then $\sna(M)$ contains a subspace isometric to $\ell_1^n$, and so, Corollary \ref{corollary:n-dim-subs} characterizes that $\sna(M)$ contains $n$-dimensional linear subspaces if and only if $M$ has more than $n$ points, providing a definitive answer to our question.

In Section \ref{section:size-subs} we study possible sizes of linear subspaces of $\sna(M)$ for metric spaces $M$, and in particular, if there can be infinite-dimensional or even non-separable subspaces. We show in Proposition \ref{proposition:size-subspaces} that, actually, any Banach space $Y$ is subspace of $\sna(M)$ for a suitable metric space $M$. We also tackle some kind of an inverse problem: how ``small'' a metric space $M$ can be so that a given Banach space $Y$ is a subspace of $\sna(M)$. We give a characterization for this in Theorem \ref{theorem:inverse-problem}. We also study similar questions for some particular classes of metric spaces. In particular, we show in Theorem \ref{proposition:sigma-compact} that if $M$ is a $\sigma$-precompact pointed metric space, then any closed linear subspace $Y$ of $\sna(M)$ is separable and isomorphically polyhedral. We also show in Proposition \ref{proposition:containing-[0,1]} that in any metric spaces that contain $[0,1]$ isometrically, such closed linear subspaces can be chosen to be infinite-dimensional.

Finally, in Section \ref{section:questions}, we exhibit some remarks and questions that remain open for now despite our attempts to solve them.

\section{Finite-dimensional subspaces}\label{section:finite-dim-subs}

In this section, we will study the existence of $n$-dimensional linear subspaces in $\sna(M)$, where $M$ is a pointed metric space and $n\in\mathbb{N}$. Our main result from the section states that if $M$ contains at least $2^n$ points (in particular, if $M$ is infinite), then $\sna(M)$ contains an isometric copy of $\ell_1^n$ (see Theorem \ref{theorem:finite-dimensional-subspace}). This provides a shocking contrast when compared to the classical theory of norm-attaining functionals, where Rmoutil showed that an infinite-dimensional Banach space $X$ introduced by Read satisfied that $\operatorname{NA}(X, \mathbb{R})$ has no $2$-dimensional subspaces (see \cite{Rmoutil17}). In order to prove our main result in this direction, we need a bit of preparatory work. 

First of all, recall that if a finite pointed metric space $M$ has exactly $n>1$ distinct points, for some $n\in\mathbb{N}$, then $\lip(M)=\sna(M)$ is an $(n-1)$-dimensional Banach space. 

\begin{remark}
Note that, in general, we cannot claim that if a Banach space $Y$ is a linear subspace of $\sna(K)$ for some metric space $K$, then $Y$ is also linearly isometric to a subspace of $\sna(M)$ for metric spaces $M$ containing $K$ as a subspace. One may be tempted to use McShane's extension theorem in order to try to get such a result, but the extensions do not behave well like a linear subspace in general. However, the well-behaving norm $\|\cdot\|_1$ will allow us to get a result in this direction, as Lemma \ref{lemma:Att_1-v2} below shows.
\end{remark}

\begin{lemma}\label{lemma:Att_1-v2}
Let $M$ be a pointed metric space such that for some subspace $K$ of $M$, $\sna(K)$ contains a linear subspace isometrically isomorphic to $\ell_1^n$ for some $n\in\mathbb{N}$. Then, $\sna(M)$ also contains a linear subspace isometrically isomorphic to $\ell_1^n$.
\end{lemma}

\begin{proof}
Indeed, let $E\subset \lip(K)$ be a linear isometric copy of $\ell_1^n$ consisting of strongly norm-attaining functionals. Then, there are $f_1,\ldots,f_n\in S_E\subset S_{\lip(K)}$ such that for all $a_1,\ldots,a_n\in \mathbb{R}$,
$$\left\| \sum_{k=1}^n a_k f_k\right\| = \sum_{k=1}^n|a_k|.$$
Let $g_1,\ldots,g_n\in S_{\lip(M)}$ be norm-preserving extensions of $f_1,\ldots,f_n$ respectively. Then, by the triangle inequality, for all $a_1,\ldots,a_n\in\mathbb{R}$,
$$\left\| \sum_{k=1}^n a_k g_k\right\| \leq \sum_{k=1}^n|a_k|.$$
On the other hand, there is a pair of different points $t_1,t_2\in K$ at which $\sum_{k=1}^n a_kf_k$ attains its norm strongly. This gives us
\begin{align*}
\left\|\sum_{k=1}^n a_kg_k\right\| &\geq \frac{\left|\left(\sum_{k=1}^n a_kg_k\right)(t_1) - \left(\sum_{k=1}^n a_kg_k\right)(t_2)\right|}{\rho(t_1, t_2)}\\
&= \frac{\left|\left(\sum_{k=1}^n a_kf_k\right)(t_1) - \left(\sum_{k=1}^n a_kf_k\right)(t_2)\right|}{\rho(t_1, t_2)}=\sum_{k=1}^n |a_k|,
\end{align*}
so $\left\|\sum_{k=1}^n a_kg_k \right\|=\sum_{k=1}^n |a_k|$, and the norm is attained strongly.
\end{proof}

In particular, if we were able to embed $\ell_1^n$ spaces isometrically in $\sna(M)$ for a finite pointed metric space $M$, we could use the previous lemma to obtain the same result for all metric spaces containing $M$. 

In the recent works \cite{KMO20} and \cite{OO20}, the existence of $\ell_1^n$ and $\ell_1$ subspaces of Lipschitz-free spaces was studied in depth, providing an answer to \cite[Question 2]{CJ17}. This has proven to be an important tool in our case, as we will use the cited below first half of \cite[Theorem 14.5]{KMO20}  in the proof of our main result. 

\begin{lemma}[{\cite[Theorem 14.5]{KMO20}}]\label{lemma:ostrovskii}
For every $n\in\mathbb{N}$, if a pointed metric space $M$ contains $2n$ elements, then $\mathcal{F}(M)$ contains a $1$-complemented subspace isometric to $\ell_1^n$.
\end{lemma}

Recall that it is not true in general that if $Y$ is a subspace of a Banach space $X$, then $Y^*$ is isometric to a subspace of $X^*$; however, the  scenario is different if $Y$ is $1$-complemented.

\begin{lemma}\label{lemma:duality-1}
Let $X$ be a Banach space that contains a $1$-complemented subspace $Y$. Then $Y^*$ embeds isometrically as a subspace of $X^*$.
\end{lemma}

\begin{proof}
Let $P: X\rightarrow Y$ be a norm-one projection. Consider the mapping $U:Y^*\rightarrow X^*$ such that for all $y^*\in Y^*$, $U(y^*):=y^*\circ P$, that is, for all $x\in X$, $U(y^*)(x):=y^*(P(x))$. Then $U$ is an isometric embedding, as desired. Indeed, just note that for each $y^*\in Y^*$, we have
$$
\| U(y^*)\|=\sup_{x\in B_X} \|U(y^*(x))\|=\sup_{x\in B_X}\|y^*(P(x))\|=\sup_{y\in B_Y}\|y^*(y)\|=\|y^*\|. \qedhere
$$
\end{proof}

Finally, recall the following well-known result, for which it is sufficient to consider the span of $n$ vectors in $\ell_\infty^{2^{n-1}}$ with $\pm 1$ coordinates, built analogously to the Rademacher functions on $[0, 1]$.

\begin{lemma}\label{lemma:ell-1-n-subs-ell-inf-p}
If $n\in\mathbb{N}$, then $\ell_1^n$ is isometric to a subspace of $\ell_\infty^{2^{n-1}}$.
\end{lemma}

We now have all the necessary tools for the proof of the main result of the section, which is the following theorem.

\begin{theorem}\label{theorem:finite-dimensional-subspace}
Let $n>1$ be a natural number, and let $M$ be a pointed metric space with at least $2^n$ distinct points. Then, there exists a linear subspace of $\sna(M)$ which is isometrically isomorphic to $\ell_1^n$.
\end{theorem}

\begin{proof}
First of all, consider a metric subspace $K$ of $M$ containing exactly $2^n$ distinct points. By Lemma \ref{lemma:ostrovskii}, $\mathcal{F}(K)$ contains a $1$-complemented subspace isometric to $\ell_1^{2^{n-1}}$. Recall that $\mathcal{F}(K)^*$ is isometric to $\operatorname{Lip}_0(K)$, and that $(\ell_1^{2^{n-1}})^*$ is isometric to $\ell_\infty^{2^{n-1}}$, so by Lemma \ref{lemma:duality-1}, $\operatorname{Lip}_0(K)=\operatorname{SNA}(K)$ contains a subspace isometric to $\ell_\infty^{2^{n-1}}$. Applying Lemma \ref{lemma:ell-1-n-subs-ell-inf-p} now we may deduce that $\operatorname{SNA}(K)$ contains a subspace isometric to $\ell_1^n$ as well. Finally, by Lemma \ref{lemma:Att_1-v2}, $\sna(M)$ also contains a subspace isometric to $\ell_1^n$.
\end{proof}

\begin{corollary}\label{corollary:inf-M-finite-subs}
If $M$ is an infinite pointed metric space and $n\in\mathbb{N}$, then $\sna(M)$ contains an $n$-dimensional subspace isometric to $\ell_1^n$.
\end{corollary}

\begin{corollary}\label{corollary:n-dim-subs}
Let $n\in\mathbb{N}$. For a pointed metric space $M$, the following statemets are equivalent:
\begin{enumerate}
\item $\sna(M)$ contains $n$-dimensional linear subspaces.
\item $M$ contains at least $n+1$ points.
\end{enumerate}
\end{corollary}

\begin{remark}
In the first preprint version of the document, we provided a fully constructive proof of Theorem \ref{theorem:finite-dimensional-subspace} for $n=2$. In order to do this, for any given metric space $M$ with exactly $4$ points, we provided two Lipschitz functionals $f_1$ and $f_2$ such that $\operatorname{span}(f_1, f_2)$ is a subspace of $\sna(M)$ isometric to $\ell_1^2$, and then we applied Lemma \ref{lemma:Att_1-v2} for bigger metric spaces. At that preprint we left the case $n>2$ as an open question despite conjecturing it to be true after our attempts to solve it. From that result, using the relation between $\sna(M)$ and $\operatorname{NA}(\mathcal{F}(M), \mathbb{R})$ and the fact that if $X^*$ contains an isometric copy of $\ell_\infty^2$ consisting of norm-attaining functionals, then $X$ contains an isometric copy of $\ell_1^2$, we deduced as a corollary that if $M$ contains at least $4$ points, then $\mathcal{F}(M)$ contains an isometric copy of $\ell_1^2$, a result which we believed to be new. However, Mikhail Ostrovskii kindly pointed us out the existence of the works \cite{KMO20,OO20}, where a more general result than our corollary was demonstrated, namely \cite[Theorem 14.5]{KMO20}, and this has allowed us to significantly improve our original results by providing a definitive answer to our question. We are deeply indebted to Mikhail Ostrovskii for this.

It seems to us that the original direct construction of $f_1$ and $f_2$ mentioned above may be of independent interest, so we permit ourselves to give it below without proof.

Let $M$ be a pointed metric space consisting in exactly $4$ points. We enumerate the elements of $M$ as $x_1, x_2, x_3, x_4$ in such a way that $x_1=0$ and 
\begin{equation*}\label{eq:condition-4-points}
\rho(x_1,x_4)+\rho(x_2,x_3)=\min\left\{\rho(\alpha,\beta)+\rho(\gamma,\delta):\, M=\{\alpha, \beta, \gamma, \delta\}\right\}.
\end{equation*}
Then the requested  $f_1, f_2 \in \lip(M)$ whose span is isometric to $\ell_1^2$ can be chosen as follows:
\begin{align*}
f_1&:\, \begin{cases}
f_1(x_1)=0,\\
f_1(x_2)=\rho(x_1,x_4)-\rho(x_2,x_4),\\
f_1(x_3)=\rho(x_1,x_4)-\rho(x_2,x_4)+\rho(x_2,x_3),\\
f_1(x_4)=\rho(x_1,x_4),
\end{cases}\\
f_2&:\, \begin{cases}
f_2(x_1)=0,\\
f_2(x_2)=\rho(x_1,x_2),\\
f_2(x_3)=\rho(x_1,x_2)-\rho(x_2,x_3),\\
f_2(x_4)=\rho(x_1,x_4).
\end{cases}
\end{align*}
\end{remark}

\section{Size of subspaces}\label{section:size-subs}

We start this section by showing that there exist metric spaces $M$ for which $\sna(M)$ contains ``big'' Banach subspaces. Actually, any Banach space $Y$ can be subspace of $\sna(M)$ for a suitable metric space $M$. 

\begin{proposition}\label{proposition:size-subspaces}
If $Y$ is a Banach space, then it is a subspace of $\sna(B_{Y^*})$.
\end{proposition}

\begin{proof}
Let $Y$ be any Banach space. Consider the metric space $B_{Y^*}$. For each $y\in Y$, let $\delta_y:B_{Y^*}\rightarrow \mathbb{R}$ be the evaluation map $\delta_y(y^*):=y^*(y)$, for all $y^*\in B_{Y^*}$. For each $y\in Y$, there exists some $y^*\in B_{Y^*}$ such that $y^*(y)=\| y \|$. It is immediate to check that $\delta_y$ is in $\lip(B_{Y^*})$ with Lipschitz constant $\| y\|$, and that it attains its norm strongly at the pair $(0, y^*)$. Therefore, $Y$ is a subspace of $\sna(B_{Y^*})$.
\end{proof}

A natural question arises now: given a Banach space $Y$, how small can a metric space $M$ be so that $Y$ is a linear subspace of $\sna(M)$? From the previous proposition, it is clear that if $Y$ has separable dual, then $M$ can be chosen to be separable. 

What if $Y^*$ is not separable? For instance, we have seen in Theorem \ref{theorem:finite-dimensional-subspace} that if $M$ is an infinite pointed metric space, then $\sna(M)$ contains isometrically all the $\ell_1^n$ spaces as linear subspaces, so it is natural to wonder if it also contains, say, $\ell_1$. However, this is not the case in general, as we are about to see. Theorem \ref{theorem:inverse-problem} below shows that separability of $Y^*$ actually characterizes the possibility of $M$ being separable. In order to prove it, we rely on the concept of James boundary introduced in Section \ref{section:introduction} (see \cite[Definition 3.118]{FHHMZ11}) and also on the following result by Gilles Godefroy.

\begin{proposition}[{Godefroy, \cite[Corollary 3.125]{FHHMZ11}}]\label{cor:sep-james-boundary}
Let $X$ be a Banach space. If $X$ has a separable James boundary, then $X^*$ is separable.
\end{proposition}

\begin{theorem}\label{theorem:inverse-problem}
For a Banach space $Y$, the following assertions are equivalent.
\begin{enumerate}
\item[(1)] There is a separable pointed metric space $M$ and a closed linear subspace $Z \subset \lip(M)$ such that $Z$ is isometric to $Y$ and $Z \subset \sna(M)$.
\item[(2)] There is a separable Banach space $X$ and a closed linear subspace $Z_1 \subset X^*$ such that $Z_1$ is isometric to $Y$ and $Z_1 \subset \na(X,\mathbb{R})$.
\item[(3)] $Y^*$ is separable.
\end{enumerate}
\end{theorem}

\begin{proof}
(1) implies (2): it is sufficient to consider $X = \mathcal F(M)$ and use the identification of $\lip(M)$ with $X^*$. With this identification $Z \subset \lip(M)$ identifies with a subspace of $Z_1 \subset X^*$ and all elements of $Z_1$ remain to be norm-attaining as elements of $X^*$.

(2) implies (3): Assume that such a separable Banach space $X$ exists, denote $J:X\rightarrow X^{**}$ the canonical embedding of $X$ into its bidual and $R: X^{**} \rightarrow Z_1^*$ the natural restriction operator. The condition $Z_1 \subset \na(X,\mathbb{R})$ means that for every $f\in Z_1$ there is $x\in B_X$ is such that $f(x)=\|f\|$, so in other words  $((R\circ J)(x))(f) = \|f\|$. Consequently, $(R\circ J)(B_X)$ is a separable James boundary of $Z_1$, so $Z_1^*$ must be separable by Godefroy's result (see Proposition \ref{cor:sep-james-boundary}), that is, $Y^*$ must be separable.

(3) implies (1): If $Y^*$ is separable, take $M=B_{Y^*}$ and apply Proposition \ref{proposition:size-subspaces}. 
\end{proof}

Therefore, there exist infinite metric spaces $M$ such that $\sna(M)$ does not contain linear subspaces isometrically isomorphic to $\ell_1$, for instance.

\begin{remark}
Note that a direct proof that (2) implies (1) in Theorem \ref{theorem:inverse-problem} can be achieved by considering $M = B_X$. In this case, the operator $U$ that maps each $f \in X^*$ to its restriction on $M$ is an isometric embedding with the property that if $f$ was norm-attaining then $U(f)$ is strongly norm-attaining on $M$. So the subspace $Z := U(Z_1)$ is what we are looking for.
\end{remark}

The next Theorem \ref{proposition:sigma-compact} shows in a similar way that if $M$ is ``small'' then the restrictions on Banach subspaces in $\sna(M)$ happen to be much stronger. In the proof we will use \cite[Corollary 2.2]{Fonf2015}.

\begin{proposition}[{\cite[Corollary 2.2]{Fonf2015}}] \label{cor:fonf}
If \(X\) has a boundary that can be covered by a set of the form \(\bigcup_{j=1}^\infty \cconv^{w^*}(K_j)\), where each \(K_j\) is countably infinite and \(w^\ast\)-compact, then \(X\) is isomorphically polyhedral.
\end{proposition}

\begin{corollary}\label{cor:separable}
If $X$ has a boundary that can be covered by a countable number of compact sets, then $X$ is separable and isomorphically polyhedral.
\end{corollary}

\begin{proof}
Let the boundary $W$ of $X$ be covered by $\bigcup_{j\in\mathbb{N}} W_j$ with $W_j$ compact sets. Then, the boundary is separable, so by Godefroy's result (Proposition \ref{cor:sep-james-boundary}), $X^*$ is separable, and then $X$ is separable as well.

According to \cite[Proposition 1.e.2]{lts}, every compact subset $W_j$ is included in a subset of the form $\cconv \{x^*_{j,k}\}_{k \in \N}\subset X^*$ where $\|x^*_{j,k}\| \xrightarrow[k \to \infty]{} 0$. Thus, the boundary $W$ has the property from Proposition \ref{cor:fonf}, so \(X\) is isomorphically polyhedral. 
\end{proof}

Remark that the same result follows from an ``internal'' characterization from \cite{Fonf90}.

\begin{theorem}\label{proposition:sigma-compact}
Let $M$ be a $\sigma$-precompact pointed metric space, then all Banach subspaces in $\sna(M)$ are separable and isomorphic to polyhedral spaces.
\end{theorem}

\begin{proof}
Let $\{M_n\}_{n=1}^{+\infty}$ be a sequence of precompact sets such that $M\subset \bigcup_{n=1}^{+\infty} M_n$. For each $n\in\mathbb{N}$, denote $\Delta_n = \{\delta_x: x \in M_n\} \subset \lip(M)^*$. By our assumption, each $\Delta_n$ is precompact in $\lip(M)^*$. Then $\overline{\mathrm{aconv}} (\Delta_n - \Delta_m)$ is compact for every $n,m\in\mathbb{N}$. The set
$$\Mol = \left\{\frac{\delta_t-\delta_{\tau}}{\rho(t,\tau)}:\, t \neq \tau\in M\right\} \subset \bigcup_{m,n,k\in\mathbb{N}} k \cdot \overline{\mathrm{aconv}} (\Delta_n - \Delta_m)$$ is covered by a countable number of norm-compact sets. 

Let $Y \subset \sna(M)$ be a Banach space. Denote $R: \lip(M)^* \to Y^*$ the natural restriction operator. Then by the continuity of $R$, $R(\Mol)$ is covered by a countable number of norm-compact sets as well. The set $R(\Mol) \bigcap S_{Y^*}$ forms a boundary for $Y$ (by the definition of strong norm-attainment),  so the statement follows from Corollary \ref{cor:separable}.
\end{proof}

Note that every compact space and every $\mathbb{R}^n$, with $n\in\mathbb{N}$, is $\sigma$-compact. In particular, every linear subspace in $\sna([0,1])$ is separable and isomorphically polyhedral. It is worth noting that such subspaces can be infinite-dimensional, as we will see in Example \ref{ex:[0,1]}.

Since all Lipschitz functions are absolutely continuous, one can identify (see for instance \cite[Example 1.6.5]{Weaver99}) the space $\lip([0,1])$ isometrically with the space $L_\infty([0,1])$, where the isometric isomorphism between them is just the differenciation operator (which exists almost everywhere):
\begin{align*}
    U: \lip([0,1)] & \rightarrow L_\infty([0,1])\\
    f &\mapsto U(f)=f'.
\end{align*}

It is clear from this and \cite[Lemma 2.2]{KMS16} that $U(\sna([0,1]))$ is the subset of $L_\infty([0,1])$ consisting on functions that attain their norm $\|\cdot\|_\infty$ throughout an interval with non-empty interior. We get the following result.

\begin{example}\label{ex:[0,1]}
If $M=[0,1]$, then $\sna(M)$ contains linear subspaces isometrically isomorphic to $c_0$.
\end{example}

\begin{proof}
Consider the set $A$ of functions $g:[0,1]\rightarrow \mathbb{R}$ such that there exists some $a=(a_1, a_2, \ldots)\in c_0$ such that
$$\left\{\begin{array}{l}
g(x) = a_k,\quad \text{if } x\in \left[ 1-\frac{1}{k}, 1-\frac{1}{k+1}\right),\ k=1,2,\ldots \\
g(1) = 0  
\end{array}\right.$$

Then $(A, \| \cdot \|_\infty)$ is a linear subspace of $U(\sna([0,1]))$ which is isometrically isomorphic to $c_0$, and that finishes the proof.
\end{proof}

Naturally, the previous example remains true if one changes $[0,1]$ with $[a,b]$, where $a<b\in\mathbb{R}$. 

Finally, the next result shows that one can extend the existence of $c_0$ in $\sna([0,1])$ to $\sna(M)$ for any pointed metric space $M$ that contains $[0,1]$ isometrically (for instance, any normed space).

\begin{proposition}\label{proposition:containing-[0,1]}
If $M$ is any pointed metric space containing $[0,1]$ isometrically, then $\sna(M)$ contains linear subspaces isometrically isomorphic to $c_0$.
\end{proposition}

\begin{proof}
Let $M$ and $Z$ be metric spaces such that $Z\subset M$. Assume that there exists some projection $F:M\rightarrow Z$ with Lipschitz constant $1$, that is:
$$\begin{cases}
F(z)=z,\quad &\text{for all $z\in Z$},\\
\rho(F(a), F(b)) \leq \rho(a, b),\quad &\text{for all $a,b\in M$}.
\end{cases}$$
Let $T: \lip(Z)\rightarrow \lip(M)$ be such that for all $f\in \lip(Z)$, $T(f):=f\circ F$. Thus, for all $x\in M$, we have $(T(f))(x)=f(F(x))$. It is clear that $T$ is linear. Moreover, $T(\lip(Z))$ is a subspace of $\lip(M)$. Hence any linear subspace of $\lip(Z)$ yields a subspace of $\lip(M)$.

All that remains is to note that if $X$ is any metric space containing $[0,1]$ isometrically, then the mapping $F$ exists. Indeed, the identity operator $\operatorname{Id}$ on $[0,1]$ is a Lipschitz function with constant $1$, and by McShane's extension theorem, it can be extended to the whole $X$ preserving its Lipschitz constant.
\end{proof}

In particular, this holds for all normed spaces. This should be once more compared with the classical theory of norm-attaining functionals, where there exist Banach spaces $X$ such that $\text{NA}(X, \mathbb{R})$ does not have $2$-dimensional subspaces (see \cite{Rmoutil17}).

\section{Open Questions and remarks}\label{section:questions}

We have seen in Theorem \ref{theorem:finite-dimensional-subspace} that if $M$ is any infinite pointed metric space, then $\sna(M)$ contains linear subspaces of every possible finite dimension. However, we do not know if every infinite pointed metric space $M$ satisfies that $\sna(M)$ contains infinite-dimensional linear subspaces, and Theorem \ref{theorem:inverse-problem} gives us some restrictions.

\begin{question}\label{question:inf-dim-subs}
Is it true that for every infinite complete pointed metric space  $M$ the corresponding $\sna(M)$ contains infinite-dimensional closed (or at least non-closed) linear subspaces?
\end{question}

Note that for Question \ref{question:inf-dim-subs}, there is no chance to have isometric copies of $\ell_1$ as a tool for its solution in general, since if the space $M$ is separable, the candidates for linear subspaces of $\sna(M)$ need to have separable dual (see Theorem \ref{theorem:inverse-problem}). Actually, in all the examples of infinite $M$ that we were able to analyze in depth, $\sna(M)$ contains isomorphic copies of $c_0$, but we do not know if this is true in general.

\begin{question}\label{question:c0-subs}
Is it true that for every infinite complete pointed metric space  $M$ the corresponding $\sna(M)$ contains an isomorphic copy of $c_0$?
\end{question}

Let us comment that, as it was mentioned in Section \ref{subsection:motivation}, in \cite[Theorem 3.2]{CDW16} it was shown that if $M$ is an infinite metric space, then $\lip(M)$ contains a linear subspace isomorphic to $\ell_\infty$, and an isometric version was achieved later in \cite[Theorem 5]{CJ17}, however, if one examines the proofs, there is no chance for them to be applicable to $\sna(M)$ in general, and actually, we refer once more to Theorem \ref{theorem:inverse-problem}.

It has been shown in Theorem \ref{proposition:sigma-compact} that if $M$ is a $\sigma$-precompact pointed metric space, then all subspaces of $\sna(M)$ are separable and isomorphically polyhedral. However, the authors do not know if the same holds for the weaker small ball property.

\begin{question}\label{question:sbp}
Let $M$ be a pointed metric space with the small ball property. Is it true that all subspaces of $\sna(M)$ are separable and isomorphically polyhedral?
\end{question}

\vspace{\baselineskip}

\noindent \textbf{Acknowledgments.} The authors are grateful to Miguel Mart\'in and Mikhail Ostrovskii for useful discussions on the topic. This research was partially done when the second author was visiting V.~N.~Karazin Kharkiv National University (Kharkiv, Ukraine), and he is very thankful for the hospitality that he received there and he wishes his Ukranian friends and colleagues the best of luck in the terrible and unfair circumstances that their country is going through at the time of writing this text.


\end{document}